\documentclass[12pt]{amsart}

\usepackage{amsmath}
\usepackage{amsthm}
\usepackage{amsfonts}
\usepackage{comment}
\usepackage{amssymb,amsrefs}
\usepackage{hyperref}
\usepackage{mathtools}
\usepackage{bm}
\usepackage[margin=1in]{geometry}

\usepackage{mathrsfs}
\usepackage{enumerate}

\usepackage{todonotes}

\newtheorem{theorem}{Theorem}[section]

\newtheorem{lemma}[theorem]{Lemma}

\theoremstyle{definition}

\newtheorem*{theorem*}{Theorem}
\newtheorem*{proposition*}{Proposition}
\newtheorem*{lemma*}{Lemma}

\theoremstyle{remark}
\newtheorem*{remark}{Remark}

\numberwithin{equation}{section}

\newcommand{\CC}{\mathbb C}
\newcommand{\FF}{\mathbb F}
\newcommand{\QQ}{\mathbb Q}
\newcommand{\OO}{\mathcal O}
\newcommand{\ZZ}{\mathbb Z}

\newcommand{\RR}{\mathbb R}
\newcommand{\Norm}{\bm N}

\makeatletter
\newcommand*{\defeq}{\mathrel{\rlap{%
                     \raisebox{0.3ex}{$\m@th\cdot$}}%
                     \raisebox{-0.3ex}{$\m@th\cdot$}}%
                     =}
\makeatother
\renewcommand{\aa}{\mathfrak a}
\newcommand{\pp}{\mathfrak p}
\renewcommand{\qq}{\mathfrak q}
\newcommand{\mm}{\mathfrak m}
\newcommand{\eps}{\varepsilon}

\begin{document}

\title[Linnik's Theorem for Sato-Tate Laws on CM Elliptic Curves]{Linnik's Theorem for Sato-Tate Laws on \\ Elliptic Curves with Complex Multiplication}

\date{\today}

\author[Evan Chen]{Evan Chen}
\address{Department of Mathematics, Massachusetts Institute of Technology, \mbox{Cambridge, MA 02139}}
\email{\href{mailto:evanchen@mit.edu}{{\tt evanchen@mit.edu}}}

\author[Peter S. Park]{Peter S. Park}
\address{Department of Mathematics, Princeton University, Princeton, NJ 08544}
\email{\href{mailto:pspark@math.princeton.edu}{{\tt pspark@math.princeton.edu}}}

\author[Ashvin A. Swaminathan]{Ashvin A. Swaminathan}
\address{Department of Mathematics, Harvard College, \mbox{Cambridge, MA 02138}}
\email{\href{mailto:aaswaminathan@college.harvard.edu}{{\tt aaswaminathan@college.harvard.edu}}}

\begin{abstract}
	Let $E/\QQ$ be an elliptic curve with complex multiplication (CM), and for each prime $p$ of good reduction, let $a_E(p) = p + 1 - \#E(\FF_p)$ denote the trace of Frobenius. By the Hasse bound, $a_E(p) = 2\sqrt{p} \cos \theta_p$ for a unique $\theta_p \in [0, \pi]$. In this paper, we prove that the least prime $p$ such that $\theta_p \in [\alpha, \beta] \subset [0, \pi]$ satisfies
    \[ p \ll \left(\frac{N_E}{\beta - \alpha}\right)^A, \]
    where $N_E$ is the conductor of $E$ and the implied constant and exponent $A > 2$ are absolute and effectively computable. Our result is an analogue for CM elliptic curves of Linnik's Theorem for arithmetic progressions, which states that the least prime $p \equiv a \pmod q$ for $(a,q)=1$ satisfies $p \ll q^L$ for an absolute constant $L > 0$.
\end{abstract}
\maketitle

\section{Introduction}
Let $E$ be an elliptic curve over $\QQ$, and for each prime $p$, let $\#E(\FF_{p})$ be the number of rational points of $E$ over the finite field $\FF_{p}$. Taking $a_E(p) =p + 1 - \# E(\FF_p)$ to be the trace of Frobenius as usual, we recall the following important result of Hasse, which holds when $E$ has good reduction at $p$:
\begin{equation*}
	\left|a_E(p)\right| \leq 2 \sqrt {p}.
\end{equation*}
It follows that for each prime $p$, there is a unique angle
$\theta_{p} \in [0, \pi]$ (which we call the ``Sato-Tate'' angle) such that $a_p = 2 \sqrt {p} \cos \theta_{p}$.
For a fixed elliptic curve $E$,
it is natural to study the distribution of the angles $\theta_p$ as $p$ ranges across the primes at which $E$ has good reduction. The now-proven Sato-Tate Conjecture provides an asymptotic for this distribution that depends on whether or not $E$ has complex multiplication (CM). While the CM case was established by Hecke, the non-CM case was recently proven in~\cite{tater} by Barnet-Lamb, Geraghty, Harris, and Taylor.
\begin{theorem*}
	[Sato-Tate Conjecture]
	Fix an elliptic curve $E/\QQ$, and let $I = [\alpha, \beta] \subset [0,\pi]$ be a subinterval. Then we have that
    \begin{equation*}
        \lim_{x \to \infty} \frac{\#\{p\leq x:\theta_p\in I\}}{\#\{p \leq x\}}
        = \begin{cases}
        	\displaystyle\int_I\frac{2}{\pi}\sin^2\theta~d\theta
            & \text{ if $E$ is non-CM,} \\[1.5em]
            \dfrac{\delta_I}{2} + \dfrac{\beta - \alpha}{2\pi}
            & \text{ if $E$ is CM}
    	\end{cases}
    \end{equation*}
    where $\delta_I = 1$ if $\pi/2 \in I$ and $\delta_I = 0$ otherwise.
\end{theorem*}
Because the Sato-Tate conjecture provides an equidistribution result for the angles $\theta_p$ in a given subinterval $I \subset [0, \pi]$, it is natural to ask whether one can determine the \emph{least} prime $p$ such that $\theta_p \in I$. In this paper, we address the CM case of this question by proving the following theorem:
\begin{theorem}
	\label{thm:main_simple}
    Let $E/\QQ$ be a CM elliptic curve of conductor $N_E$. There exists a prime $p$ such that $\theta_p \in I$ and
    \[ p \ll \left(\frac{N_E}{\beta - \alpha}\right)^A, \]
    where the implied constant and exponent $A > 2$ are absolute and effectively computable.
\end{theorem}

Observe that Theorem~\ref{thm:main_simple} is analogous to Linnik's Theorem, which provides an upper bound on the least prime in an arithmetic progression. Specifically, Linnik showed in~\cites{linnik1, linnik2} that the least prime $p \equiv a \pmod q$, for relatively prime integers $a$ and $q$, satisfies $p \ll q^L$ (where the implied constant and the exponent $L > 0$ are absolute and effectively computable). This analogy between Theorem~\ref{thm:main_simple} and Linnik's Theorem is reasonable to expect; indeed, the least prime $p$ with $\theta_p \in I$ should grow inversely with the length of $I$ and should depend in some way on the arithmetic properties of $E$ (such as its conductor), just as the least prime $p$ in an arithmetic progression modulo $q$ should grow with $q$.
\begin{remark} The non-CM analogue of Theorem~\ref{thm:main_simple} was proven by Lemke Oliver and Thorner in~\cite{thorner}. Their bound depends on the number of symmetric-power $L$-functions of $E$ that are known to have analytic continuations and functional equations of the usual type.

Also, it is well-known that for a given elliptic curve $E/\QQ$, the traces of Frobenius $a_E(p)$ are the $p^\mathrm{th}$ Fourier coefficients of a weight $2$ newform of level $N_E$. One can thus formulate this problem in the more general context of newforms of even weight $k\geq 2$ with complex multiplication; the proof is essentially the same as the proof of Theorem~\ref{thm:main_simple}.
\end{remark}
The rest of this paper is organized as follows. Section~\ref{defs} presents an introduction to the analytic theory of CM elliptic curves and $L$-functions associated to Hecke Gr\"ossencharaktere, which are the fundamental tools that we employ in our proof of Theorem~\ref{thm:main_simple}. Then, Section~\ref{proof} employs the tools developed in Section~\ref{defs} to give a detailed proof of Theorem~\ref{thm:main_simple}.

\section{CM Elliptic Curves and Hecke $L$-Functions}\label{defs}

In this section, we provide a brief description of the relevant facts about CM elliptic curves over $\QQ$ and $L$-functions of Hecke Gr\"ossencharaktere that are employed in our proof of Theorem~\ref{thm:main_simple}; a standard reference is \cite{Iwaniec199709}. Note that throughout the rest of the paper, all implied constants are absolute unless otherwise specified.

Let $K/\QQ$ be an algebraic number field, and let $\mm \subset \OO_K$ be a nonzero integral ideal. Let $\xi$ denote a Hecke Gr\"ossencharakter over $K$ of modulus $\mm$ and frequency $k$.
When $K$ is an imaginary quadratic field, every Hecke Gr\"ossencharakter 
can be thought of as the product of a ray-class character $\chi : (\OO_K/\mm)^\ast \to S^1$
with an angle character $\chi_\infty : \CC^\ast \to S^1$, where $S^1 = \{z \in \CC : |z| = 1\}$.
(Here, by frequency of $\xi$, we mean the frequency of $\chi_\infty$. See \cite{iwaniec} for details.)
The Hecke $L$-function $L(s,\xi)$ associated to $\xi$ is defined as the Euler product
\[
	L(s,\xi) =
    \prod_{\pp \subset \OO_K}
    (1-\xi(\pp)\Norm(\pp)^{-s})^{-1},
\]
which converges absolutely for $\sigma > 1$. Hecke showed that the above product can be meromorphically continued to the entire complex plane, giving an $L$-function whose degree equals $[K : \QQ]$. Furthermore, he proved that $L(s, \xi)$ is entire if $\xi$ is nontrivial and that $L(s, \xi)$ has a simple pole at $s = 1$ when $\xi$ is trivial.

As described in~\cite{pid} and~\cite{advanced}, the theory of Hecke Gr\"ossencharaktere is fundamental to the study of CM elliptic curves. Let $E/\QQ$ be an elliptic curve of conductor $N_E$, and suppose that $E$ has complex multiplication by the ring of integers $\OO_K$ of a number field $K/\QQ$ with absolute discriminant $|d_K|$. Recall that in this case, $K$ is necessarily an imaginary quadratic field of class number $1$, so that $\OO_K$ is a principal ideal domain.
For prime ideals $\pp \subset \OO_K$ at which $E$ has good reduction, set
\[ a_E(\pp) =\Norm(\pp) + 1 - \# E(\FF_\pp) \]
where $\FF_\pp \defeq \OO_K / \pp$. Then, the Hasse bound tells us that
\begin{equation*}
	\left|a_E(\pp)\right| \leq 2 \sqrt {\Norm(\pp)}.
\end{equation*}
Thus, for each prime ideal $\pp \subset \OO_K$ at which $E$ has good reduction, we can define $\theta_{\pp} \in [0, \pi]$ such that $a_\pp = 2 \sqrt {\Norm(\pp)} \cos \theta_{\pp}$. Now, consider a totally multiplicative map $\xi_E$ that is defined on unramified prime ideals $\pp \subset \OO_K$ by
\[ \xi_E(\pp) = \exp\left(\pm i \theta_{\pp}\right) \]
where the symbol ``$\pm$'' indicates a sign that depends on $\pp$ and $E$.
It is a well-known result of Weil (see~\cite{weil}) that if the signs $\pm$ are chosen appropriately for each $\pp$, then $\xi_E$ is a Hecke Gr\"ossencharakter over $K$. We note that $\xi_E$ has frequency $1$, and as discussed in~\cite{ken}, the modulus $\mm$ of $\xi_E$ has norm $\Norm(\mm) = N_E/|d_K|$. For $k \in \ZZ \setminus \{ 0 \}$, we denote by $\xi_E^k$ the map defined by $\xi_E^k(\aa) \defeq \xi_E(\aa)^k$ for nonzero ideals $\aa \subset \OO_K$; observe that this is a Hecke Gr\"ossencharakter of modulus $\mm$ and frequency $k$.

Taking the analytic conductor $\qq(s, \xi)$ of an $L$-function $L(s, \xi)$ to be defined as in Equation 5.7 of~\cite{iwaniec}, it is easy to deduce the following useful bound on the analytic conductor of $L(s,\xi_E^k)$:
\begin{equation}\label{anaconda}
\log\qq(s, \xi_E^k) \ll \log \big( (|s|+3) \cdot \Norm(\mm) \cdot k \big).
\end{equation}
We devote the remainder of this section to presenting a few relevant results on the distribution of nontrivial zeros of Hecke $L$-functions; we will apply these results in Section~\ref{proof} to $L(s, \xi_E^k)$. The following lemma, which is adapted from Theorem 5.10 in~\cite{iwaniec}, provides a zero-free region for Hecke $L$-functions over quadratic fields of class number $1$.
\begin{lemma}
   	\label{lem:zfr}
	Let $K$ be a quadratic number field of class number $1$,
    let $\mm \subset \OO_K$ be a nonzero integral ideal,
    and let $\xi$ be a Hecke Gr\"ossencharakter modulo $\mm$.
    Then $L(s, \xi)$ has at most one zero in the region
	\[
        \sigma \ge 1 - \frac{c_1}%
        {\log \qq(it,\xi)}
    \]
	for some absolute constant $c_1 > 0$.
    The exceptional ``Siegel zero'' can only exist if $\xi$ is
    a real quadratic character and is necessarily both real and simple.
\end{lemma}
Note that the region defined in Lemma~\ref{lem:zfr} is free of zeros when the Hecke Gr\"ossencharakter is trivial or has infinite order. Since the character $\xi_E^k$ is trivial if $k = 0$ and has infinite order if $k \neq 0$, we need not consider Siegel zeros in applying Lemma~\ref{lem:zfr} to $L(s, \xi_E^k)$. The next lemma, which is adapted from part (1) of Proposition 5.7 in~\cite{iwaniec}, provides an estimate on the vertical distribution of zeros of Hecke $L$-functions over quadratic fields:
\begin{lemma}
    \label{lem:vdz}
	Retain the setting of Lemma~\ref{lem:zfr}.
    For any $t \ge 2$,
    the number of zeros $\rho$ of $L(s, \xi)$ with $\gamma \in [t-1,t+1]$ is less than
    \[ c_2 \log \qq(it,\xi) \]
    for some absolute constant $c_2$.
\end{lemma}
A key input into the proof of Linnik-type theorems is a logarithm-free zero-density estimate. In our proof of Theorem~\ref{thm:main_simple}, we will employ the following estimate, which we have adapted from~\cite{fogels}:
\begin{lemma}
   	\label{lem:lfzde}
    Fix an integer $H \geq 1$,
    an imaginary quadratic number field $K$ of class number $1$,
    and a nonzero integral ideal $\mm \subset \OO_K$. Consider the product
	\[ L(s; \mm, H) = \prod_{\xi} L(s, \xi), \]
	where $\xi$ ranges over all Hecke Gr\"ossencharaktere with modulus $\mm$
    and frequency at most $H$.
    Let $N(\lambda,T)$ denote the number of zeros of $L(s; \mm, H)$ that
    lie in the rectangle
	\[
        1-\lambda < \beta < 1
        \quad\text{and}\quad
        \left\lvert \gamma \right\rvert \le T.
    \]
	Then there exists an absolute constant $c_3 \in (0,1)$ and
    an absolute constant $c_4$ such that if $\lambda \in (0,c_3)$
    and $T \geq \Norm(\mm)(1 + H)$, then
    \[ N(\lambda, T) \leq T^{c_4\lambda}. \]
\end{lemma}

\begin{remark}
Similar zero-density estimates were obtained by Koval'\v{c}ik in~\cite{lithuaniastrikesagain}. 
These density estimates are unlikely to produce Linnik-type theorems because they are not logarithm-free. However, they do have applications in studying primes of the form $p=a^2+b^2$ where $|b| < p^{1/4+\epsilon}$ and in producing an analogue of the Bombieri-Vinogradov theorem for primes $p=a^2+b^2$ where $\arg(a+bi)$ lies in a given sector. We thank Professor Jean-Pierre Serre for introducing us to this paper.
\end{remark}

\section{Proof of Theorem~\ref{thm:main_simple}}\label{proof}

In this section, we provide a complete proof of the main result in this paper, namely Theorem~\ref{thm:main_simple}. Let $E/\QQ$ be an elliptic curve of conductor $N_E$ with CM by $\OO_K$, where $K$ is necessarily an imaginary quadratic field of class number 1. Recall from Section~\ref{defs} that we can associate to $E$ a Hecke Gr\"ossencharakter $\xi_E$ over $K$ of modulus $\mm \subset \OO_K$ and frequency $1$. Fix a subinterval $I = [\alpha, \beta] \subset [0,\pi]$ with indicator function denoted by $\chi_I$, and put
\begin{equation}\label{eqx} x = \frac{\Norm(\mm)}{\beta-\alpha}.\end{equation}
Notice that $x$ has a positive lower bound of $1/\pi$, and recall that
\[\frac{\Norm(\mm)}{\beta-\alpha} = \frac{N_E}{|d_K|(\beta-\alpha)} \leq
\frac{N_E}{\beta-\alpha}.  \]
Thus, to prove Theorem~\ref{thm:main_simple}, it suffices to show that if $x$ is sufficiently large,
we can pick a constant $A > 2$ so that there exists a prime $p \ll x^A$ with $\theta_p \in I$.
The method we employ in this section is based on the work of Graham and Jutila
on computing explicit Linnik constants (see~\cites{scan,graham}) as well as that of Kaufman (see~\cite{lith}).

\subsection{Initial Setup of the Proof}
Let $A > 2$ be a sufficiently large absolute constant.
Let $R : (0,\infty) \to \RR$ be supported on $[x^{A-2}, x^A]$.
Consider the sum $S$ defined by
\begin{equation}
	\label{eq:reeq}
	S \defeq \sum_{\substack{\pp \subset \OO_K \\ f_\pp = 1}}
	\frac{\log \Norm(\pp) R(\Norm(\pp))\chi_I(\theta_\pp)}{\Norm(\pp)}.
\end{equation}
Here, the sum is taken over unramified prime ideals $\pp$ (henceforth all sums over primes will implicitly be taken over unramified primes). By $f_\pp$ we mean the \emph{inertial degree} of $\pp$,
which is the degree of $\OO_K / \pp$ as an $\FF_p$-vector space (recall that $\Norm(\pp)= p^{f_\pp}$).
In our case, since $K$ is a quadratic field, we have $f_\pp \in \{1,2\}$.
We will show that $S > 0$.

As in~\cite{graham}, we construct the function $R(y)$ by means of a kernel. For $s \in \CC$, define a kernel\footnote{The kernel, as defined in \S7 of~\cite{graham} is missing a factor of $s$ in the denominator. We have corrected the kernel in our definition of $K(s)$.} $K(s)$ by
\[ K(s) \defeq x^{\frac{A-2}{2} \cdot s} \left( \frac{x^s-1}{s\log x}\right), \]
and take the function $R(y)$ to be given by
\begin{equation}\label{eq:rawr}
	R(y) \defeq \frac{1}{2\pi i}
	\int_{2-i\infty}^{2+i\infty} K(s)^2 y^{-s} \, ds.
\end{equation}
As stated in~\cites{graham,scan}, the function $R(y)$, as defined above, vanishes outside of the interval $[x^{A-2}, x^A]$ and satisfies $R(y) \ll (\log x)^{-1}$ when $y \in [x^{A-2}, x^A]$. We will utilize the following bound on our function $K(s)$:\footnote{This bound, as stated in (22) of~\cite{graham}, has an extraneous minus sign in the exponent of $x$. We have corrected the statement in Lemma~\ref{lem:kernel}.}

\begin{lemma}[Graham,~\cite{graham}]
	\label{lem:kernel}
	Let $B_1 = A - 2$. For $\sigma < 0$, we have that
	\[ \left\lvert K(s) \right\rvert^2
		\le
		x^{B_1\sigma} \min \left( 1, \frac{4}{\left\lvert s \right\rvert^2 (\log x)^2} \right).
	\]
\end{lemma}
\smallskip

\subsection{Estimating $S$} In order to rephrase our problem into one that concerns the Hecke Gr\"ossencharaktere $\xi_E^k$, we use the following lower approximation to $\chi_I$ with symmetric, compactly supported Fourier coefficients:

\begin{lemma}
   	\label{lem:trig}
    Let $I = [\alpha, \beta] \subset [0, \pi]$ be a subinterval, and let $M \in \ZZ_{> 0}$. There exists a trigonometric polynomial
    \[ S_{I,M}(\theta) = \sum_{|n| \leq M} b_n \exp(i n \theta) \]
    satisfying the following properties: For all $\theta \in [0, \pi]$, we have $S_{I,M}(\theta) \leq \chi_I(\theta)$, and for all $n \in \{-M, \dots, M\} \setminus \{0\}$ we have that $b_n = b_{-n}$ and that
   \begin{equation} \label{fouriers}
    \left|b_0 - \frac{\beta - \alpha}{\pi}\right| \leq \frac{2}{M+1} \text{ and } |b_n| \leq \left(\frac{2}{M+1}+\min\left\{\frac{\beta - \alpha}{\pi},\frac{2}{\pi|n|}\right\}\right). \end{equation}
\end{lemma}
\begin{proof}
The lemma follows by modifying the Beurling-Selberg minorant polynomials (see~\cite[\S1.2, p.~5-6]{montgomery} for a formal definition of these polynomials) to be even and periodic modulo $2\pi$.
\end{proof}

We are now in a position to estimate the indicator function $\chi_I$ of the interval $I = [\alpha, \beta] \subset [0,\pi]$ in terms of the Hecke Gr\"ossencharaktere $\xi_E^k$. We set $M = x^{1+\eps}$ for an absolute $\eps \in (0, 1/2)$. From Lemma~\ref{lem:trig}, we find that for each unramified prime ideal $\pp \subset \OO_K$,
\[
	\chi_I(\theta_\pp)
    \ge
    \sum_{|k| \le M} b_k \exp(ik\theta_\pp)
	=
	b_0 + \sum_{\substack{|k| \le M \\ k \neq 0}} b_k \xi_E^k(\pp),
\]
where the Fourier coefficients $b_k$ satisfy the conditions specified in (\ref{fouriers}). Applying this lower approximation to $\chi_I$ to (\ref{eq:reeq}), we obtain the following estimate on $S$:
\begin{equation*}
	S = \sum_{f_\pp=1} \frac{\log \Norm(\pp)R(\Norm(\pp))\chi_I(\theta_\pp)}{\Norm(\pp)}
    \ge \sum_{f_\pp=1} \sum_{|k|\le M} b_k \frac{\log \Norm(\pp)R(\Norm(\pp))\xi_E^k(\pp)}{\Norm(\pp)}.
\end{equation*}
Since $R(y)$ is nonzero for only finitely many integers $y$, the sum over primes in the right-hand-side of the above inequality is a finite sum. Thus, we can exchange the order of summation to conclude that
\begin{equation}
	\label{eq:sumest}
	S \geq \sum_{|k| \le M} b_kS_k, 
	\quad \quad S_k = 
    \sum_{f_\pp=1} \frac{\log \Norm(\pp)R(\Norm(\pp))\xi_E^k(\pp)}{\Norm(\pp)}.
	\end{equation}
In what follows, we denote the inner sum on the right-hand-side of \eqref{eq:sumest} by $S_k$.

\subsection{Estimating $S_k$}\label{ell}
To estimate $S_k$ using our knowledge of the Hecke $L$-function $L(s, \xi_E^k)$, we will introduce for every $k \in \{-M, ..., M\}$ an integral $I_k$ defined as follows:
\[
	I_k
	\defeq
	\frac{1}{2\pi i} \int_{2-i\infty}^{2+i\infty} K(s)^2 \left(-\frac{L'}{L}(s+1, \xi_E^k)\right) \, ds
\]
Evaluating the logarithmic derivative of $L(s+1, \xi_E^k)$, we find that
\begin{equation}
	\label{eq:log}
	-\frac{L'}{L}(s+1, \xi_E^k)
    = \sum_{\aa}
    \frac{\Lambda_K(\aa) \xi_E^k(\aa)}{\Norm(\aa)^{s+1}},
\end{equation}
where $\Lambda_K$ is the von Mangoldt function over the number field $K$, defined as
\[
	\Lambda_K(\aa)
	=
	\begin{cases}
		\log \Norm(\pp) & \text{if } \aa = \pp^m, \\
		0 & \text{otherwise}.
	\end{cases}
\]
Substituting (\ref{eq:log}) into the definition of the integral $I_k$ and integrating term by term, we obtain the following series representation of $I_k$:
\begin{equation}\label{eq:seriesrep}
    I_k
    = \sum_{\aa}
    \frac{\Lambda_K(\aa) R(\Norm(\aa)) \xi_E^k(\aa)}{\Norm(\aa)}.
\end{equation}

Recall from (\ref{eq:seriesrep}) that $I_k$ can be expressed as a sum over prime powers, whereas the desired sum $S_k$ is a sum over primes $\pp$ with $f_\pp=1$.
We bound the difference between $I_k$ and $S_k$ as follows.
\begin{lemma}
    For any $k$ we have $|I_k-S_k| = O(x^{-2})$.
\end{lemma}
\begin{proof}
    Recall the fact that $R(y) = 0$ for $y \notin [x^{A-2}, x^A]$,
    and moreover that $R(y) \ll (\log x)^{-1}$ for $y \in [x^{A-2}, x^A]$.
    Thus, when $\Norm(\aa)\in [x^{A-2}, x^A]$, we have that
    \[
    \left|\frac{\Lambda_K(\aa) R(\Norm(\aa)) \xi_E^k(\aa)}{\Norm(\aa)}\right|
    \ll \frac{\log(x^A) \cdot (\log x)^{-1}}{x^{A-2}} = \frac{A}{x^{A-2}} .
    \]
    The number $N$ of nonzero terms in the sum (\ref{eq:seriesrep}) corresponding to ideals $\aa = \pp^k$ with $k > 2$ is at most twice the number of prime powers $\le x^A$, so we have that $N \ll x^{A/2}$.
    Moreover, the number of prime ideals with $f_\pp = 2$ is also at most $x^{A/2}$
    (since the norm of such a prime ideal is necessarily a perfect square).
    Thus, the difference between $S_k$ and $I_k$ can be bounded as follows:
    \[ |S_k - I_k| \ll \frac{A}{x^{A-2}} \cdot x^{A/2} \ll x^{-2} \]
    provided that $A \ge 8$.
\end{proof}

On the other hand, we can evaluate the integral $I_k$ by shifting the contour from $\sigma = 2$ to $\sigma = -5/4$.\footnote{In performing an analogous calculation, Kaufman shifts the contour to $\sigma = -3/2$ (see~\cite{lith}), but this is not possible because for a Hecke Gr\"ossencharakter $\xi$ with frequency 1, the logarithmic derivative of $L(s+1,\xi)$ has a pole at $s = -3/2$.}
To this end, we prove the following lemma:
\begin{lemma}\label{cont}
We have that
\[ I_k  = \delta(k) - {\sum_{\rho}}^k K(\rho-1)^2 + O(x^{-2}) \]
where the superscript ``$k$'' on the sum indicates that the sum is taken over nontrivial zeros $\rho$ of $L(s, \xi_E^k)$ and where $\delta(k)$ denotes the Kronecker delta function.
\end{lemma}
\begin{proof}
    Consider the truncated integral $I_k(T)$ defined for $T>0$ by
    \[
        I_k(T)
        \defeq
        \frac{1}{2\pi i} \int_{2-iT}^{2+iT} K(s)^2
        \left(-\frac{L'}{L}(s+1, \xi_E^k)\right) \, ds,
    \]
    where $T$ does not coincide with the ordinate of a zero of $L(s, \xi_E^k)$.
    We want to shift the contour from \[ \sigma = 2 \quad\text{to}\quad \sigma = -5/4. \]
    In doing so, the nontrivial zeros of $L(s+1, \xi_E^k)$,
    which occur when $s+1 = \rho$, contribute residues that sum to $-{\sum_\rho}^k K(\rho - 1)^2$.
    When $k = 0$, we know that $L(s, \xi_E^k)$ has a simple pole when $s + 1 = 1$, which contributes a residue of $\delta(k)$. Moreover, if $k = 0$, then $L(s, \xi_E^k)$ has a trivial zero at $s = -1$, which contributes a residue that is bounded by
    \[ \ll \frac{x^{2-A}(1 - x^{-1})^2}{(\log x)^2} \ll x^{-2} \]
    provided that $A>4$. It is easy to check that the integrand of $I_k$ has no other poles in the range $-5/4 \leq \sigma \leq 2$. Thus, by the Residue Theorem, we have that
    \[
    	I_k(T) = \delta(k) - \sum_{\rho} K(\rho -1)^2 + O(x^{-2}) + \frac{1}{2\pi i}
        \int_{\Gamma_T} K(s)^2\left(-\frac{L'}{L}(s+1, \xi_E^k)\right) \, ds
    \]
    where $\Gamma_T$ is the rectangular path consisting of the three legs
    \[ 	2-iT
        \;\longrightarrow\; -\frac54-iT
        \;\longrightarrow\; -\frac54+iT
        \;\longrightarrow\; 2+iT. \]
	In order to evaluate the above integral, we require a bound on the logarithmic derivative of $L(s,\xi_E^k)$. To this end, one can obtain from (\ref{anaconda}), Lemma~\ref{lem:vdz}, and part (2) of Proposition 5.7 in~\cite{iwaniec} that for $s$ satisfying $-\frac14\le\sigma\le 3$ and $|t| = T$ sufficiently large, we have
		\[ \left\lvert \frac{L'}{L}(s, \xi_E^k) \right\rvert = O\left(\left(\log k(T+3)\right)^2 \right). \]
	Note that the condition of having $T$ sufficiently large can be removed if $\sigma = -5/4$, because $-5/4$ is bounded away from $0$, $1$, and all local parameters of $L(s, \xi_E^k)$ at infinity. This is important  for estimating the integral along the vertical leg $\sigma = -\frac{5}{4}$ of $\Gamma_T$.
	Now, the integral along the first leg (horizontal leg at $t = -T$) is bounded in absolute value by
			\[
				\ll
				\sup_{\substack{-\frac54\le\sigma\le2\\t = -T}}
				\left\lvert x^{(A-2)s}\left(\frac{x^s - 1}{s\log x}\right)^2  \frac{L'}{L}(s+1, \xi_E^k) \right\rvert
				\ll \frac{x^{2(A-2)} (x^2+1)^2}{T^2} \left(\log k(T+3)\right)^2,
			\]
		which vanishes as $T \to \infty$. By an analogous argument, the integral along the third leg (horizontal leg at $t = T$) vanishes as $T \to \infty$. Finally, the integral along the second leg (the vertical leg at $\sigma = -5/4$) is bounded in absolute value by
			\begin{align*}
				&\ll \sup_{\substack{\sigma = -\frac{5}{4}\\|t| \le T}}\left|  x^{(A-2)s}\left(\frac{x^s - 1}{\log x}\right)^2  \right|
				\int_{-T}^{T} \frac{\left( \log k(|t|+3)\right)^2}{ \left|-\frac{5}{4}+it\right|^2} \, dt  \\
				&\ll x^{-\frac{5}{4}(A-2)}\left(\frac{x^{-\frac{5}{4}} + 1}{\log x}\right)^2 \int_{-T}^{T} \frac{\left(\log k(|t|+3)\right)^2}{ \left|-\frac{5}{4}+it\right|^2} \, dt.
			\end{align*}
			Notice that as $T \to \infty$, the integral in the above expression converges by the $p$-test. Therefore, provided that $A > 6$, we have that the above term is $\ll x^{-2}$ in the limit as $T \to \infty$, which proves the lemma.
\end{proof}

\subsection{Estimating the Sum over Zeros}\label{zerosum}

We now combine our Fourier estimate of $S$ with our estimate of $S_k$. By \eqref{eq:sumest}, Lemma~\ref{lem:trig} and the results of Section~\ref{ell}, we have
\begin{align}
    S \geq \sum_{|k| \le M} b_k S_k & = \sum_{|k| \le M} b_k (I_k + O(x^{-2})) \nonumber\\
    & = \sum_{|k| \le M} b_k \left(\delta(k)- {\sum_{\rho}}^k K(\rho-1)^2 + O(x^{-2})\right) \nonumber \\
    & = \frac{\beta - \alpha}{\pi}-\sum_{|k| \le M} b_k {\sum_{\rho}}^k K(\rho-1)^2 + O(x^{-2}) \cdot \sum_{|k| \leq M}b_k \nonumber\\
    & = \frac{\beta - \alpha}{\pi}-\sum_{|k| \le M} b_k {\sum_{\rho}}^k K(\rho-1)^2 + o(x^{-1}), \label{popit}
\end{align}
where we have used our choice of $M = x^{1+\varepsilon}$. We now wish to provide a tight bound on the sum in (\ref{popit}).
We now prove the central lemma in our estimate:
\begin{lemma}
     We have that for sufficiently large $x$,
     \[ \sum_{|k| \le M} {\sum_\rho}^k
     \left\lvert K(\rho-1) \right\rvert^2 < \frac{9}{10}. \]
\end{lemma}
\begin{proof}
   First, notice that since $\xi_E^k$ has infinite order for any $k \neq 0$, we may apply Lemma~\ref{lem:zfr} without consideration of Siegel zeros. Now, let $B_1 = A - 2$ and $M = x^{1+\eps}$ as before, set $T = M^2 =x^{2+2\eps}$, and let $B_2 = B_1 - (2+2\eps)c_4$ (see Lemma~\ref{lem:lfzde} for the definition of $c_4$) and assume $B_2 > 0$ (by selecting $A$ large enough).
    We begin by computing the following Stieltjes integral over $\lambda$
    using the bounds given by Lemmas~\ref{lem:lfzde} and~\ref{lem:kernel}, the former of which will apply if we take $x$ sufficiently large so that $T > \Norm(\mm)(M+1)$:
    \begin{align}
        \int_a^b x^{-B_1\lambda} \; dN(\lambda, T)
        &= \left.x^{-B_1\lambda} N(\lambda, T)\right|_a^b
        + B_1\log x \int_a^b N(\lambda, T) x^{-B_1\lambda} \; d\lambda \nonumber \\
        &\le \left\lvert x^{-B_1a} \cdot T^{c_4a} \right\rvert +  \left\lvert x^{-B_1b} \cdot T^{c_4b} \right\rvert
        + \left\lvert B_1\log x \cdot \int_a^b T^{c_4\lambda} x^{-B_1\lambda} \; d\lambda \right\rvert \nonumber\\
        &= x^{-B_2a} + x^{-B_2b}
        + B_1\log x \left\lvert \int_a^b x^{-B_2\lambda} d\lambda \right\rvert \nonumber\\
        &\leq \frac{B_1 + B_2}{B_2} \left(x^{-B_2a} + x^{-B_2b}\right). \label{brumly}
    \end{align}
    We will now bound the contribution of zeros in the rectangle defined by $1 - c_3 < \beta < 1$ and $|\gamma| < T$ using (\ref{brumly}). We first need to choose $a,b$ appropriately. Applying the zero-free region stated in Lemma~\ref{lem:zfr} to $L(s, \xi_E^k)$, we can pick
        \[ a = \frac{c_1}{\log \qq(iT, \xi_E^k)} \quad\text{and}\quad b \to \infty. \]
    Given our choices of $T$ and $M$ as well as the fact that $|k| \leq M$, we deduce from (\ref{anaconda}) $\log \qq(iT, \xi_E^k) < C\log x$ for some absolute constant $C > 0$. Substituting these choices of $a,b$ into (\ref{brumly}), we find that for sufficiently large $x$, the contribution of zeros in this rectangle
    is at most $B_3$, where
    \begin{equation}
        B_3 \defeq \frac{B_1 + B_2}{B_2}
        \left( \exp\left(-\frac{B_2c_1}{C}\right) \right).
        \label{eq:sub_one}
    \end{equation}
    Next, we bound the contribution of zeros in the rectangle $0 < \beta < 1 - c_3$
    and $\left\lvert \gamma \right\rvert < T$; we show that it yields a negligible contribution of $o(1/x)$.
    The contribution of each zero with $\beta < 1 - c_3$ is at most $x^{-B_1c_3}$
    by Lemma~\ref{lem:kernel}.
    Therefore, if we sum over zeros in vertical strips
    $[t-1,t+1]$ for $t = 0, 1, \dots, T$
    and appeal to Lemmas~\ref{lem:vdz} and~\ref{lem:kernel},
    we obtain the bound
    \begin{align*}
        \sum_{|k| \leq M} \sum_{\substack{0 < \beta < 1-c_3 \\ |\gamma| < T}}
        \left\lvert K(\rho-1) \right\rvert^2
        &\ll M \cdot \sum_{t=0}^T x^{-B_1c_3} \log x \\
        &\ll MT x^{-B_1c_3} \log x \\
        &\ll x^{3+3\eps-B_1c_3} \log x,
    \end{align*}
    which is $o(1/x)$ as long as $B_1 = A-2$ is sufficiently large. Finally, we will show that the contribution of zeros with $|\gamma| \ge T$ is also negligible.
By Lemmas~\ref{lem:vdz} and~\ref{lem:kernel}, this contribution is
\begin{align*}
	\sum_{|k| \le M} {\sum_{\rho}}^k \frac{4x^{-B_1(1-\beta)}}{\left\lvert \rho-1 \right\rvert^2 (\log x)^2}
	&\ll \sum_{|k| \le M} {\sum_{\rho}}^k \frac{1}{\left\lvert \rho-1 \right\rvert^2 (\log x)^2} \\
	&\ll \frac{1}{(\log x)^2} \sum_{|k| \leq M } \sum_{t > T} \frac{\log(k t)}{t^2} \\
	&\ll \frac{M \log T}{T(\log x)^2},
\end{align*}
	which is $o(1/x)$. In the last step above, we used the fact that $k \leq M \leq T$ and bounded the sum over $t$ with an integral. To obtain the lemma, we simply need to select $A$ in such a way that $B_3 < 9/10$, which is possible because $B_3$ can be made arbitrarily small by taking $A $ sufficiently large.
\end{proof}

\subsection{Completing the Proof}
For convenience, put $\tau = \frac{\beta-\alpha}{\pi} \le 1$.
Observing that $\frac{2}{M+1} = O(x^{-1-\eps}) =  o(x^{-1})$ and recalling our bound on $S$, we see that
\begin{eqnarray*}
    S &\ge &\tau - \sum_{|k| \le M} b_k {\sum_{\rho}}^k K(\rho-1)^2 + o(x^{-1}) \\
	&\ge & \tau - \left( \tau + \frac{2}{M+1} \right)
	 \sum_{|k|\le M} {\sum_{\rho}}^k |K(\rho-1)|^2
	- o(x^{-1}) \\
	&\ge &\tau - \left( \tau + \frac{2}{M+1} \right)
	\left( \frac{9}{10} + o(x^{-1}) \right)
	- o(x^{-1})\\
	&\ge & \frac{1}{10}\tau - o( x^{-1} ).
\end{eqnarray*}
As $x = \pi\Norm(\mm)/\tau$, it follows that $S > 0$ for $x$ sufficiently large.
Using our definitions of $S$ in \eqref{eq:reeq} and $R$ in \eqref{eq:rawr}, it follows that there exists a $\pp$ such that $f_\pp = 1$, $\theta_\pp \in [\alpha, \beta]$
and $\Norm(\pp) \in [x^{A-2}, x^A]$. Since $f_\pp = 1$, we can write $\pp=(p)$ for a rational prime $p$. We then have that $\theta_\pp = \theta_p$, from which we deduce that
\[ \theta_p = \theta_\pp \in [\alpha, \beta] \quad\text{and}\quad p \le x^A. \]
This completes the proof of the main result, Theorem~\ref{thm:main_simple}.

\begin{remark}
Notice that if a rational prime $p$ is inert in $\OO_K$, then $a_E(p) = 0$, so that $\theta_p = \pi/2$. Thus, whenever $\pi/2 \in I$, we have that all inert primes $p \nmid N_E$ satisfy $\theta_p \in I$.
Thus, in this case, the bound in Theorem~\ref{thm:main_simple} can be improved substantially. In particular, the bound no longer depends on the length $\beta - \alpha$ of the interval $I$.
\end{remark}

\section*{Acknowledgments}

\noindent This research was supervised by Ken Ono at the Emory University Mathematics REU and was supported by the National Science Foundation (grant number DMS-1250467). We would like to thank Ken Ono and Jesse Thorner for offering their advice and guidance and for providing many helpful discussions and valuable suggestions on the paper. We would also like to thank Professor Jean-Pierre Serre for pointing us to the reference~\cite{lithuaniastrikesagain}.

\bibliographystyle{amsxport}
\bibliography{biblio}

\end{document}